\documentclass[12pt, twoside, leqno]{article}

\usepackage{amsmath,amsthm}
\usepackage{amssymb}
\usepackage{enumerate}

\newtheorem{thm}{Theorem}[section]
\newtheorem{cor}[thm]{Corollary}
\newtheorem{lem}[thm]{Lemma}

\newtheorem{obs}{Observation}

\theoremstyle{definition}
\newtheorem{defin}[thm]{Definition}

\newcommand*\oline[1]{%
  \vbox{%
    \hrule height 0.5pt
    \kern0.25ex
    \hbox{%
      \kern-0.2em
      \ifmmode#1\else\ensuremath{#1}\fi
      \kern -0.05em
    }
  }
}

\numberwithin{equation}{section}

\begin{document}


\baselineskip=17pt


\title{Strong Homology, Derived Limits, and Set Theory}

\author{Jeffrey Bergfalk\\
Department of Mathematics\\
Malott Hall\\
Cornell University\\
Ithaca, NY 14853-4201\\
U.S.A.\\
E-mail: bergfalk@math.cornell.edu}

\date{}

\maketitle


\renewcommand{\thefootnote}{}



\footnote{Research was supported in part by NSF grant DMS-1262019.}

\renewcommand{\thefootnote}{\arabic{footnote}}
\setcounter{footnote}{0}


\begin{abstract}
We consider the question of the additivity of strong homology. This entails isolating the set-theoretic content of the higher derived limits of an inverse system indexed by the functions from $\mathbb{N}$ to $\mathbb{N}$. We show that this system governs, at a certain level, the additivity of strong homology over sums of arbitrary cardinality. We show in addition that, under the Proper Forcing Axiom, strong homology is not additive, not even on closed subspaces of $\mathbb{R}^4$.
\end{abstract}

\section{Introduction}

The strong homology theory $\oline{H}_*$, defined for all topological spaces, has the following desirable properties:
\begin{enumerate}
\item It satisfies all the Eilenberg-Steenrod axioms on paracompact pairs $(X,A)$.
\item It is strong shape invariant.
\item It is a Steenrod-type homology theory (and therefore Alexander dual to $\check{H}^*$); it satisfies, in other words, two of the three axioms Milnor proposed to supplement Eilenberg and Steenrod's (\cite{Milnor1}, \cite {Milnor2}; see \cite{Mar1}).
 \end{enumerate}
 It remains an open question on what class of spaces it may satisfy the third of those axioms, \textit{additivity}, the condition that every mapping
 $$i:\displaystyle\bigoplus_A \oline{H}_p(X_\alpha)\rightarrow \oline{H}_p(\displaystyle\coprod_A X_\alpha)$$
 induced by inclusion maps $i_\alpha:X_\alpha\hookrightarrow\coprod_A X_\alpha$ be an isomorphism.

It was in investigation of this question that Marde\'si\v{c} and Prasolov computed the strong homology of $Y^{(k)}$, the topological sum of countably many $k$-dimensional Hawaiian earrings. They showed, in particular, that $\oline{H}_p(Y^{(k)})=\text{lim}^{k-p}\mathbf{A}$ for $0<p<k$, where $\mathbf{A}$ is an abelian pro-group indexed by $\mathbb{N}^\mathbb{N}$ (see below). For a single $k$-dimensional Hawaiian earring $X^{(k)}$, $\oline{H}_p(X^{(k)})=0$ for $0<p<k$; thus additivity requires at least that $\text{lim}^n\,\mathbf{A}=0$ for $n>0$. Marde\'si\v{c} and Prasolov then showed that the continuum hypothesis implies that $\text{lim}^1\,\mathbf{A}\neq 0$ \cite{MarPras}. Shortly thereafter, Dow, Simon, and Vaughan showed that the Proper Forcing Axiom (PFA) implies that $\text{lim}^1\,\mathbf{A}=0$ and, hence, that the vanishing of $\text{lim}^1\,\mathbf{A}$ is independent of the axioms of ZFC \cite{DSV}. This vanishing, in fact, is a question of broad interest in its own right; see \cite{T1}, for instance, and the discussion therein.

It is the purpose of this note to extend those investigations. In sections 3 and 4, we show the vanishing of $\text{lim}^2\,\mathbf{A}$ also independent of the axioms of ZFC and characterize, more generally, the higher $\text{lim}^n\,\mathbf{A}$. In particular, we show that, under PFA, strong homology is not additive, not even on the category of, e.g., closed subspaces of $\mathbb{R}^4$ (our witness in this case is $\;\oline{H}_1(Y^{(3)})$). In section 5, for $\kappa$ infinite, we let $\mathbf{A}_\kappa$ denote a pro-group analogous to $\mathbf{A}$ but indexed by $\mathbb{N}^\kappa$; we show $\text{lim}^1\,\mathbf{A}_\kappa=0$ if and only if $\text{lim}^1\,\mathbf{A}=0$. Extending, as it does, the topological significance of the system $\mathbf{A}$, this is the main theorem of the paper. In section 6, we list some open problems.

In section 2, we define our notation, the system $\mathbf{A}$, and briefly review the derived functors $\text{lim}^n$ of $\text{lim}$. This paper aims to interest readers in both homological algebra and set theory, and therefore - with a few mild exceptions in section 4 - assumes no more than a basic knowledge of either. In particular, no knowledge of forcing is presumed; the reader need only understand that the Proper Forcing Axiom, $\diamondsuit(S_1^2)$, $\text{MA}_{\aleph_1}$, and $\mathfrak{d}=\aleph_1$ (or $\aleph_2$) are prominent set-theoretic hypotheses independent of the axioms of ZFC. For more on the Proper Forcing Axiom, see in particular \cite{Moore}. For more on set theory generally, see \cite{Jech} or \cite{Kunen}. For further on $\text{lim}$ and its derived functors, see [Ma $\mathsection 11$] and \cite{Jen}.

\section{Background and Notation}

Our inverse systems all will consist of abelian groups $X_f$ and ``bonding'' homomorphisms $p_{fg}:X_g\rightarrow X_f$ for every $f\leq g$. Our index-set will typically be $\mathcal{N}=\mathbb{N}^\mathbb{N}$, ordered coordinatewise: $f\leq g$ if and only if $f(i) \leq g(i)$ for all $i\in\mathbb{N}$. Our particular focus is $\mathbf{A}=(A_f,p_{fg},\mathcal{N})$, where $$A_f=\displaystyle\bigoplus_{i\in\mathbb{N}}\mathbb{Z}^{f(i)}$$ with projection mappings $p_{fg}$. Relatedly, $\mathbf{B}=(B_f,p_{fg},\mathcal{N})$, where
$$B_f=\displaystyle\prod_{i\in\mathbb{N}}\mathbb{Z}^{f(i)}$$
We consider only level morphisms $\mathbf{F}:\mathbf{X}\rightarrow\mathbf{Y}$ among such systems, i.e., collections of functions $F_f:X_f\rightarrow Y_f$ which commute with all the bonding maps. Likewise, terms of the quotient $\mathbf{Y}/\mathbf{X}$ are of the form $Y_f/X_f$, so that
\begin{align}\label{e:1}0\rightarrow \mathbf{A}\xrightarrow{\mathbf{I}} \mathbf{B}\xrightarrow{\mathbf{Q}} \mathbf{B}/\mathbf{A}\rightarrow 0\end{align}
for example, is exact. In the language of category theory, we study the abelian category $\mathcal{A}b^\mathcal{\,N}$.

An abelian group $X$ together with $\mathbf{p}=\{p_f:X\rightarrow X_f\,|\,f\in\mathcal{N}\}$ is an inverse limit of $\mathbf{X}$ if
\begin{align}\label{e:2} p_f=p_{fg} p_g\text{ for all }f\leq g\in\mathcal{N} \end{align}
and for any $(Y,\mathbf{q})$ satisfying (\ref{e:2}) there exists a unique $q:Y\rightarrow X$ such that $\mathbf{p} q=\mathbf{q}$. Such an $X$ and $\mathbf{p}$ are unique up to isomorphism; we henceforth write $X=\text{lim}\,\mathbf{X}$ for the group alone. $X$ admits the following description:
\begin{align}\label{e:3}\text{lim}\,\mathbf{X}=\{\langle x_f\rangle\in\displaystyle\prod_{f\in\mathcal{N}}X_f\,|\,p_{fg}(x_g)=x_f\text{ for all }f\leq g\in\mathcal{N}\}\end{align}
Note that 
\begin{align*} \text{lim}\,\mathbf{B} & \cong \,\displaystyle\prod_{i\in\mathbb{N}}\,\prod_{j\in\mathbb{N}}\mathbb{Z} \\
\text{lim}\,\mathbf{A} & \cong \displaystyle\bigoplus_{i\in\mathbb{N}}\prod_{j\in\mathbb{N}}\mathbb{Z}\end{align*}
Returning to (\ref{e:3}), for $\mathbf{F}:\mathbf{X}\rightarrow\mathbf{Y}$, define $\text{lim}\,\mathbf{F}:\text{lim}\,\mathbf{X}\rightarrow\text{lim}\,\mathbf{Y}$ as the induced mapping of products. We define thereby a functor $\text{lim}:\mathcal{A}b^\mathcal{N}\rightarrow\mathcal{A}b$. We are interested in the following phenomenon: $\text{lim}$ applied to sequence (\ref{e:1}), for example, may fail to be exact. More precisely, $\text{lim}$ is left exact: $\text{lim}\,\mathbf{I}$ will be injective, but $\text{lim}\,\mathbf{Q}$ may fail to be surjective, to a degree the long exact sequence
\begin{align}\label{e:4}0\rightarrow \;\text{lim}\,\mathbf{A}\xrightarrow{\text{lim}\,\mathbf{I}} \text{lim}\,\mathbf{B}\xrightarrow{\text{lim}\,\mathbf{Q}} \text{lim}\,\mathbf{B}/\mathbf{A}
\xrightarrow{\theta_0} \text{lim}^1\,\mathbf{A}\xrightarrow{\text{lim}^1\,\mathbf{I}} \text{lim}^1\,\mathbf{B} \dots\end{align}
in some sense measures. The non-exactness of lim induces, in other words, a sequence of derived functors $\text{lim}^n$ connected, for any short exact sequence in $\mathcal{A}b^\mathcal{N}$, by a long exact sequence of abelian groups of the above form, with connecting transformations $\theta_n$. These functors $\text{lim}^n$, like $\text{lim}$, admit explicit description; see the proof of Theorem~\ref{limnA}, below. From this description, the reader may verify the following:

(i) For any constant system $\mathbf{X}=(X_f,p_{fg},\mathcal{N})$, i.e., any system with $X_f=X$ and $p_{fg}=id$ for all $f\leq g\in\mathcal{N}$, $\text{lim}^n\,\mathbf{X}=0$ for $n\geq 1$.

(ii) $\text{lim}^n\,\mathbf{B}=0$ for $n\geq 1$.

Returning to (\ref{e:4}): by (ii), $\text{lim}^1\,\mathbf{A}=0$ if and only if $\text{lim}\,\mathbf{Q}$ is surjective. To better articulate that equivalence, we introduce the following conventions, basic to all that follows:

For $f\in\mathcal{N}$, let $I_f=\{(i,j)\,|\,j\leq f(i)\}$. For $f,g\in\mathcal{N}$ write $f<^* g$ if $\{i\,|\,f(i)\nless g(i)\}$ is finite. Write $f\leq^* g$ if $\{i\,|\,f(i)\nleq g(i)\}$ is finite or, equivalently, if $I_f\subseteq^*I_g$. $\phi_f$ will denote a function of the form $I_f\rightarrow \mathbb{Z}$. Write $\phi=^*\psi$ if $\{x\in\text{dom}(\phi)\cap\text{dom}(\psi)\,|\, \phi(x)\neq\psi(x)\}$ is finite; note that this is not, in general, an equivalence relation. A collection $\Phi=\{\phi_f\,|\,f\in\mathcal{N}\}$ is \textit{coherent} if $\phi_f=^*\phi_g$ for all $f,g\in\mathcal{N}$. $\Phi$ is \textit{trivial} if there exists $\phi:\mathbb{N}^2\rightarrow\mathbb{Z}$ such that $\phi=^*\phi_f$ for all $f\in\mathcal{N}$. We may view any $\phi_f$ as an element of $B_f$; write $[\phi_f]$ for its image in $B_f/A_f$. We may view $\Phi$, likewise, as an element of $\prod_{\mathcal{N}}B_f$; writing $[\Phi]$ for $\{[\phi_f]\,|\,f\in\mathcal{N}\}$, then, $$\text{lim}\,\mathbf{B}/\mathbf{A}\cong\{[\Phi]\,|\,\Phi\text{ is coherent}\}$$
Hence $\text{lim}\,\mathbf{Q}$ is surjective if and only if every coherent $[\Phi]$ equals $\{[\phi\!\!\upharpoonright\!\!_{I_f}]\,|\,f\in\mathcal{N}\}$ for some $\phi:\mathbb{N}^2\rightarrow\mathbb{Z}$ in $\text{lim}\,\mathbf{B}$. In other words,
\begin{thm} \label{lim1A} \cite{MarPras} $\;\text{lim}^1\,\mathbf{A}=0$ if and only if every coherent family of functions $\Phi=\{\phi_f\,|\,f\in\mathcal{N}\}$ is trivial.
\end{thm}
\noindent It's this observation we generalize in section 3.

We recall, finally, the following notions from set theory. The cofinality of a partial order $P$ is $$\text{cof}(P)=\text{min}\{|Q|\,|\,\text{for all }p\in P\text{ there exists a }q\in Q\text{ with }q\geq p\}$$The cofinality of an inverse system is the cofinality of its index-set. Observe that $\text{cof}(\mathcal{N},<)=\text{cof}(\mathcal{N},<^*)$. We write $\mathfrak{d}$ for either.$$\mathfrak{b}=\text{min}\{|\mathcal{F}|\,|\,\text{for all }g\in \mathcal{N}\text{ there exists an }f\in \mathcal{F}\text{ with }f\nless^* g\}$$ Symbols $\alpha,\beta,\xi$ denote ordinals; $\kappa$ denotes a cardinal. $[\kappa]^{<\kappa}=\{y\subset\kappa\,|\,\kappa>|y|\}$. For $A\subseteq\text{dom}(f)$, $f''A=\{f(a)\,|\,a\in A\}$. A cofinal subset $C$ of $\beta$ is \textit{club} if it is closed in $\beta$ under the topology induced by the membership relation. $S\subseteq\beta$ is \textit{stationary} if it intersects all club subsets of $\beta$.

\begin{defin} $\diamondsuit(S^2_1)$ is the assertion that there exists a family $\mathcal{S}=\{S_\beta\,|\,\beta<\omega_2\text{ and cof}(\beta)=\aleph_1\}$ such that, for any $A\subset\omega_2$, the set $\{\beta\,|\,A\cap\beta=S_\beta\}$ is stationary.\end{defin}

The reader may verify that the intersection of two club subsets of $\beta$ is a club and, hence, that the intersection of a club and a stationary subset of $\beta$ is stationary; these facts and the straightforward implication $\diamondsuit(S^2_1)\Rightarrow 2^{\aleph_0}\leq\aleph_2$ play a role in the proof of Theorem~\ref{lim2A}.

\section{Characterizing higher derived limits of $\mathbf{A}$}

Let $\mathcal{N}^{[n]}=\{(f_0,\dots,f_{n-1})\in\mathcal{N}^n\, |\, f_i\leq f_j\text{ for all }i<j\}$ for $n>0$, and let $\mathcal{N}^{[0]}=\{\varnothing\}$.
Let $\vec{f}^i$ denote the $(n-1)$-tuple obtained by deleting $f_i$ from $\vec{f}\in\mathcal{N}^{[n]}$; $\phi_{\vec{f}}$ will denote a function of the form $I_{f_0}\rightarrow \mathbb{Z}$ unless $\vec{f}=\varnothing$, in which case $\phi_{\vec{f}}:\mathbb{N}\times\mathbb{N}\rightarrow\mathbb{Z}$.

\begin{defin}\label{ncoh}
A collection $\Phi=\{\phi_{\vec{f}}\,|\,\vec{f}\in\mathcal{N}^{[n]}\}$ is \emph{n-coherent} if, for all $\vec{f}\in\mathcal{N}^{[n+1]}$,
$$\phi_{\vec{f}^0}\!\!\upharpoonright\!\!_{I_{f_0}}+\displaystyle\sum_{i=1}^{n}(-1)^i\phi_{\vec{f}^i}=^{*}0.$$
\end{defin}

For readability, we henceforth write such sums, simply, as $\displaystyle\sum_{i=0}^{n}(-1)^i\phi_{\vec{f}^i}$.

\begin{defin}\label{ntriv}
$\Phi$ is \emph{n-trivial} if there exists $\{\psi_{\vec{f}}\,|\,\vec{f}\in\mathcal{N}^{[n-1]}\}$ such that, for all $\vec{f}\in\mathcal{N}^{[n]}$,
$$\displaystyle\sum_{i=0}^{n-1}(-1)^i\psi_{\vec{f}^i}=^{*}\phi_{\vec{f}}.$$
\end{defin}

\begin{thm}\label{limnA} For $n>0$, $\text{lim}^n\,\mathbf{A}=0$ if and only if every $n$-coherent $\Phi=\{\phi_{\vec{f}}\,|\,\vec{f}\in\mathcal{N}^{[n]}\}$ is $n$-trivial.
\end{thm}

\begin{proof} Define cochain complex $\mathcal{K}(\mathbf{B}): 0\rightarrow K^0(\mathbf{B})\rightarrow K^1(\mathbf{B})\rightarrow\dots$ by
$$K^n(\mathbf{B})=\displaystyle\prod_{\vec{f}\in\mathcal{N}^{[n+1]}} B_{f_0}$$
with differential $d^n:K^{n-1}(\mathbf{B})\rightarrow K^n(\mathbf{B})$ defined, for $n>0$, by
$$(d^n\Psi)(\vec{f})=\displaystyle\sum_{i=0}^{n}(-1)^i\Psi(\vec{f}^i)$$
Analogously define $\mathcal{K}(\mathbf{A})$, a subcomplex of $\mathcal{K}(\mathbf{B})$. View any $\Phi$ as in the statement of the theorem as an element of $K^{n-1}(\mathbf{B})$; observe that $\Phi$ is $n$-coherent if and only if $d^n\Phi\in K^{n}(\mathbf{A})$, and is $n$-trivial if and only if $\Phi - d^{n-1}\Psi\in K^{n-1}(\mathbf{A})$ for some $\Psi\in K^{n-2}(\mathbf{B})$.

N\"obeling and Roos independently established that, in general, $\text{lim}^n\,\mathbf{X}\cong H^n\mathcal{K}(\mathbf{X})$ (see [Ma $\mathsection 11.5$] for proof; the reader may more immediately verify that $H^0(\mathcal{K}(\mathbf{X}))=\text{lim}\,\mathbf{X}$). In particular, $\text{lim}^n\,\mathbf{A}=0$ if and only if, in $\mathcal{K}(\mathbf{A})$, every $n$-cocycle is an $n$-coboundary. Assume the latter, and take $n\geq 2$ (the case $n=1$ was Theorem~\ref{lim1A}): if $\Phi$ is $n$-coherent, then $d^n\Phi\in K^n(\mathbf{A})$ is an $n$-cocycle and hence, by assumption, equals $d^n\Upsilon$ for some $\Upsilon\in K^{n-1}(\mathbf{A})$. Since $\text{lim}^{n-1}\,\mathbf{B}=0$, cocycle $(\Phi-\Upsilon)$ equals $d^{n-1}\Psi$ for some $\Psi\in K^{n-2}(\mathbf{B})$; in other words, $\Phi - d^{n-1}\Psi\in K^{n-1}(\mathbf{A})$, i.e., $\Phi$ is $n$-trivial.

On the other hand, if every $n$-coherent $\Phi$ is $n$-trivial, take $n$-cocycle $\Upsilon\in K^n(\mathbf{A})$. Since $\text{lim}^{n-1}\,\mathbf{B}=0$, $\Upsilon=d^n\Phi$ for some $\Phi\in K^{n-1}(\mathbf{B})$. $\Phi$ is $n$-coherent, so by assumption, $\Phi-d^{n-1}\Psi\in K^{n-1}(\mathbf{A})$ for some $\Psi\in K^{n-2}(\mathbf{B})$; hence $\Upsilon=d^n(\Phi-d^{n-1}\Psi)$ is an $n$-coboundary in $\mathcal{K}(\mathbf{A})$.
\end{proof}
We will sometimes consider systems indexed by orders extending or contained in $\mathcal{N}$; the appropriate modification of definitions should in such cases be obvious.

Early indications of the relevance of set-theoretic considerations to higher derived limits were the following:

\begin{thm} \cite{Gob} \label{limkA} For any inverse system $\mathbf{X}$ of cofinality $\aleph_k$, $\text{lim}^n\,\mathbf{X}=0$ for all $n\geq k+2$.
\end{thm}

\begin{thm} \cite{Mitch} For every $k\geq 0$ there exists an inverse system $\mathbf{X}$ of cofinality $\aleph_k$ with $\text{lim}^{k+1}\,\mathbf{X}\neq 0$.
\end{thm}

\begin{cor} \label{limk+1A} If $\mathfrak{d}=\aleph_k$, then $\text{lim}^n\,\mathbf{A}=0$ for all $n\geq k+1$.
\end{cor}

\begin{proof} Immediate, by Theorem~\ref{limkA}, for $n>k+1$. Let $\mathbf{D}=(D_f,p^d_{fg},\mathcal{N})$, with $D_f=\{\phi:\mathbb{N}^2\backslash I_f\rightarrow\mathbb{Z}\, |\,\text{supp}(\phi)\text{ is finite}\}$ and $p^d_{fg}$ the inclusion map; let $\mathbf{E}=(E_f,p^e_{fg},\mathcal{N})$, with $E_f=\{\phi:\mathbb{N}^2\rightarrow\mathbb{Z}\, |\,\text{supp}(\phi)\text{ is finite}\}$ and $p^e_{fg}$ the identity. Form short exact sequence
$$0\rightarrow\mathbf{D}\rightarrow\mathbf{E}\rightarrow\mathbf{A}\rightarrow 0$$
inducing long exact sequence
$$\dots\rightarrow\text{lim}^{k+1}\,\mathbf{E}\rightarrow\text{lim}^{k+1}\,\mathbf{A}\rightarrow\text{lim}^{k+2}\,\mathbf{D}\rightarrow\dots$$
As noted, for $k\geq 0$, $\text{lim}^{k+1}\,\mathbf{E}=0$, so by Theorem~\ref{limkA}, $0=\text{lim}^{k+2}\,\mathbf{D}=\text{lim}^{k+1}\,\mathbf{A}$.
\end{proof}

By the corollary, together with the following theorem, $\mathfrak{d}=\aleph_1$ fully determines when $\text{lim}^n\,\mathbf{A}=0$.

\begin{thm} \label{NEQ} \cite{DSV} If $\mathfrak{d}=\aleph_1$ then $\text{lim}^1\,\mathbf{A}\neq0$.
\end{thm}

\section{PFA and $\text{lim}^2\,\mathbf{A}$}

By the following theorem, PFA fully determines when $\text{lim}^n\,\mathbf{A}=0$, as well - but in a different direction.

\begin{thm} \label{lim2A} If $\mathfrak{b}=\mathfrak{d}=\aleph_2$ and $\diamondsuit(S^2_1)$ then $\text{lim}^2\mathbf{A}\neq 0$.
\end{thm}

\begin{cor} \label{PFA} Under the Proper Forcing Axiom, $\text{lim}^n\mathbf{A}\neq 0$ if and only if $n=2$.
\end{cor}

\begin{proof}[Proof of Corollary~\ref{PFA}] Among the consequences of PFA:
\begin{enumerate}
\item $\mathfrak{d}=\aleph_2$ (\cite{V}, \cite{Bek}). So by Corollary~\ref{limk+1A}, $\text{lim}^n\mathbf{A}=0$ for $n>2$.
\item $\text{lim}^1\mathbf{A}=0$ (\cite{DSV}). This and $\mathfrak{b}=\aleph_2$ follow in fact from a strictly weaker assumption, the Open Coloring Axiom, a consequence of PFA (\cite{T2}).
\item $\diamondsuit(S^2_1)$ (\cite{Baum}).
\end{enumerate} 
Theorem~\ref{lim2A} then completes the proof. 
\end{proof}

The condition $\mathfrak{b}=\mathfrak{d}=\aleph_2$ is equivalent to the existence of an $\omega_2$-scale.

\begin{defin} A $\gamma$-chain in $\mathcal{N}$ is a collection $\{f_\alpha\,|\,\alpha<\gamma\}\subset\mathcal{N}$ such that $\alpha<\beta$ implies $f_\alpha<^* f_\beta$. A $\gamma$-scale is a $\gamma$-chain which is $<^*$-cofinal in $\mathcal{N}$.
\end{defin}

Theorem~\ref{NEQ} is perhaps better understood as a ZFC phenomenon:

\begin{thm} \label{lim1AF} For any $\omega_1$-chain $\mathcal{F}$ in $\mathcal{N}$, there exists a nontrivial coherent $\Phi^{\mathcal{F}}=\{\phi_f\,|\,f\in\mathcal{F}\}$.
\end{thm}

In other words, $\text{lim}^1\,\mathbf{A}^\mathcal{F}\neq 0$, where $\mathbf{A}^\mathcal{F}=(A_f,p_{fg},\mathcal{F})$. Theorem~\ref{lim1AF} is simply a recasting of [Be 96-98], which inscribes a gap in any $\subset^*$-increasing $\omega_1$-chain of subsets of $\mathbb{N}$. Let $\mathcal{F}^*=\{g\in\mathcal{N}\,|\,g\leq^*f\text{ for some }f\in\mathcal{F}\}$; write $\Phi^\mathcal{F}$ for $\{\phi_f\,|\,f\in\mathcal{F}\}$, as above. Any coherent $\Phi^\mathcal{F}$ extends to a coherent $\Phi^{\mathcal{F}^*}$, so the theorem gives a nontrivial coherent $\Phi^\mathcal{G}$ for any $\mathcal{G}\subseteq\mathcal{N}$ of cofinality $\aleph_1$ in the $<^*$-ordering. Such $\Phi^\mathcal{G}$ admit no ``upwards'' extensions:

\begin{obs} For any $h$ with $g\leq^* h$ for all $g\in\mathcal{G}$, no nontrivial coherent $\Phi^\mathcal{G}$ extends to a coherent $\Phi^{\mathcal{G}\cup\{h\}}$. \end{obs}

\noindent For if it did, then any $\phi:\mathbb{N}\times\mathbb{N}\rightarrow\mathbb{Z}$ extending $\phi_h$ would trivialize $\Phi^\mathcal{G}$, a contradiction. This is one key to the proof below. The other is the following:

\begin{obs} If $\Phi_1$ and $\Upsilon_1$ 2-trivialize the same $\Phi_2$ then they differ by a 1-coherent $\Psi_1$. \end{obs}

\begin{proof}[Proof of Theorem~\ref{lim2A}]
Fix an $\omega_2$-scale $\mathcal{F}=\{f_\alpha\,|\,\alpha<\omega_2\}$ and an $\mathcal{S}$ witnessing $\diamondsuit(S^2_1)$. Let $\mathcal{F}_\beta=\{f\in\mathcal{N}\,|\,f\leq^*f_\alpha\text{ for some }\alpha<\beta\}$. Let $Y_\beta=\bigcup_{f\in \mathcal{F}_\beta}\mathbb{Z}^{I_f}$ and $Y=\bigcup_{\beta<\omega_2}Y_\beta$, and fix a bijection $\rho:\omega_2\rightarrow Y$. Those $\beta$ for which $\rho''\beta=Y_\beta$ form a club subset of $\omega_2$; hence for any $\Phi_1=\{\phi_f\,|\,f\in\mathcal{N}\}$ the set $S(\Phi_1)=\{\beta\in S_1^2\,|\,\rho''S_\beta=\Phi_1\cap Y_\beta\}$ is stationary.

We show $\text{lim}^2\,\mathbf{A}\neq 0$ by constructing, in stages $\Phi_2^\beta$ $(\beta<\omega_2)$, a non-2-trivial $2$-coherent $\Phi_2$: each $\Phi_2^\beta$ will be of the form $$\{\phi_{fg}\,|\,f\leq g\leq^*f_\alpha\text{ for some }\alpha<\beta\}$$ and $\Phi_2$ will be their union. By Corollary~\ref{limk+1A}, $\text{lim}^2\,\mathbf{A}^{\mathcal{F}_\beta}=0$ for every $\beta<\omega_2$, so every 2-coherent $\Phi_2^\beta$ is $2$-trivial, and therefore extends to some 2-trivial (hence 2-coherent) $\Phi_2^{\beta+1}$. For limit $\beta$, let $\Phi_2^\beta=\bigcup_{\gamma<\beta}\Phi_2^\gamma$. At stage $\beta$, if $\rho''S_\beta$ is of the form $\{\phi_f\,|\,f\in\mathcal{F}_\beta\}$ and 2-trivializes $\Phi_2^\beta$, we extend with greater care. Since $\text{cof}(\beta)=\aleph_1$, there exists by Theorem~\ref{lim1AF} a nontrivial coherent family $\Psi_1^{\mathcal{F}_\beta}$. Take any extension $\Upsilon_1^{\beta+1}=\{\upsilon_f:f\in\mathcal{F}_{\beta+1}\}$ of $\Upsilon_\beta=\rho''S_\beta+\Psi_1^{\mathcal{F}_\beta}$. Letting $\phi_{fg}=\upsilon_g\!\!\upharpoonright_{I_f}\!\!-\,\upsilon_f$ for any $g\in\mathcal{F}_{\beta+1}\backslash\mathcal{F}_\beta$ defines a 2-coherent extension $\Phi_2^{\beta+1}$ of $\Phi_2^\beta$ which is 2-trivialized by $\Upsilon_1^{\beta+1}$.

Clearly $\Phi_2$ is 2-coherent. Suppose for contradiction that $\Phi_1$ 2-trivializes $\Phi_2$. Then for $\beta\in S(\Phi_1)$, $\Phi_1\cap Y_\beta$ and $\Phi_1\cap Y_{\beta+1}$ 2-trivialize $\Phi_2^\beta$ and $\Phi_2^{\beta+1}$, respectively. By construction, $\Upsilon_1^{\beta+1}$ also 2-trivializes $\Phi_2^{\beta+1}$. By Observation 2, then, $\Upsilon_1^{\beta+1}-(\Phi_1\cap Y_{\beta+1})$ is a coherent family extending nontrivial coherent family $\Upsilon_1^{\beta}-(\Phi_1\cap Y_\beta)=\Psi_1^{\mathcal{F}_\beta}$, contradicting Observation 1.
\end{proof}

Beginning from a model of PFA, Todorcevic forced $\text{lim}^1\,\mathbf{A}\neq 0$ while preserving MA$_{\aleph_1}$ (see \cite{T1}; note that his argument -- and hence that of the theorem below -- requires none of the large cardinal consistency strength of PFA). Forcing over his model by the analogue of the above proof (conditions are 2-coherent $\Phi_2^\beta$, ordered by inclusion), then, gives the following:

\begin{thm} Under the Proper Forcing Axiom, MA$_{\aleph_1}$ is consistent with ``$\text{lim}^n\,\mathbf{A}\neq 0$ if and only if $n\leq 2$''.
\end{thm}

\section{Relating $\text{lim}^1\mathbf{A}$ to $\text{lim}^1\,\mathbf{A}_\kappa$}

Let $\mathbf{A}_\kappa=(A_f,p_{fg},\omega^\kappa)$, where $A_f=\bigoplus_{\alpha<\kappa}\mathbb{Z}^{f(\alpha)}=\{\phi_f:I_f\rightarrow\mathbb{Z}\,|\,\text{supp}(f)\text{ is}$ $\text{finite}\}$ and $p_{fg}:\phi_f\mapsto\phi_f\!\!\upharpoonright_{I_g}$, as before. $\mathbf{A}_\kappa$ generalizes $\mathbf{A}$ both in form ($\mathbf{A}_\omega=\mathbf{A}$) and in content: $\bar{H}_p(Y^{(k)})=\text{lim}^{k-p}\,\mathbf{A}_\kappa$ for $Y^{(k)}$ the disjoint union of $\kappa$ many $k$-dimensional Hawaiian earrings when $0<p<k$. We show the following relation:

\begin{thm} \label{limAkap} $\text{lim}^1\,\mathbf{A}=0$ if and only if $\text{lim}^1\,\mathbf{A}_\kappa=0$ for all infinite $\kappa$.
\end{thm}

Replacing $\mathcal{N}$ with $\omega^\kappa$ in Definitions~\ref{ncoh} and \ref{ntriv}, the arguments of Theorem~\ref{limnA} apply equally to $\text{lim}^n\,\mathbf{A}_\kappa$, so one direction of the theorem is clear: if every $n$-coherent family $\{\phi_{\vec{f}}\,|\,\vec{f}\in(\omega^\kappa)^{[n]}\}$ is $n$-trivial, so too must be every $n$-coherent family $\{\phi_{\vec{f}}\,|\,\vec{f}\in(\omega^\omega)^{[n]}\}$. In other words:
\begin{obs}  For $n> 0$, $\text{lim}^n\,\mathbf{A}_\kappa=0$ implies $\text{lim}^n\,\mathbf{A}=0$. \end{obs}
For the other direction of Theorem~\ref{limAkap}, we assume $\text{lim}^1\,\mathbf{A}=0$, fix a coherent family $\Phi=\{\phi_f\,|\,f\in\omega^\kappa\}$ and show it trivial. This we'll argue by transfinite induction on $\kappa$. The argument separates into the two cases $\text{cof}(\kappa)=\omega$ and $\text{cof}(\kappa)>\omega$. The hypothesis in all cases is that $\text{lim}^1\,\mathbf{A}_\lambda=0$ for $\lambda<\kappa$; hence $\Phi\!\!\!\upharpoonright_{x}\,=\{\phi_f\!\!\!\upharpoonright_{I_f\cap(x\times\omega)}\!|\;f\in\omega^\kappa\}$ is trivial for any $x\in[\kappa]^{<\kappa}$. We'll want to measure the failure of various $\phi$ to trivialize $\Phi$; for this the notation $e(\phi,\psi)=\{\alpha\,|\,\phi(\alpha,i)\neq\psi(\alpha,i)\text{ for some }i\}$ will be useful.

\begin{proof}

Case 1: $\text{cof}(\kappa)=\omega$. Fix $\{\beta_j\,|\,j<\omega\}$ cofinal in $\kappa$, with $\beta_0=0$. Let $L_j=[\beta_j,\beta_{j+1})$ and fix, for all $j<\omega$, some $\phi_j:L_j\times\omega\rightarrow\mathbb{Z}$ trivializing $\Phi\!\!\upharpoonright_{L_j}$. For all $\alpha<\kappa$, there's a unique $j(\alpha)$ such that $\alpha\in L_{j(\alpha)}$. Define $\phi:\kappa\times\omega\rightarrow\mathbb{Z}$ by $\phi(\alpha,i)=\phi_{j(\alpha)}(\alpha,i)$. Let $\text{err}(\phi)=\{f\in\omega^\kappa\,|\,\phi_f\neq^*\phi\}$, i.e., $\text{err}(\phi)$ collects those $f$ such that $e(\phi_f,\phi)$ is infinite. Note that $e(\phi_f,\phi)$ is countable for every $f$, and that $\text{err}(\phi)=\varnothing$ if and only if $\phi$ trivializes $\Phi$.

Say $x$ \textit{bounds} a collection $\mathcal{C}\subset P(\kappa)$ if $c\subset^* x$ for all $c\in\mathcal{C}$. For any $x\in[\kappa]^{<\kappa}$ bounding $\{e(\phi_f,\phi):f\in\text{err}(\phi)\}$, it is our induction hypothesis that some $\psi: x\times\omega\rightarrow\mathbb{Z}$ trivializes $\Phi\!\!\upharpoonright_{x}$. Define $\phi': \kappa\times\omega\rightarrow\mathbb{Z}$:
\begin{displaymath}
   \phi'(\alpha,i) = \left\{
     \begin{array}{lr}
       \psi(\alpha,i) & \alpha\in x\\
       \phi(\alpha,i) & \text{otherwise}
     \end{array}
   \right.
\end{displaymath}
Observe that $\phi'$ trivializes $\Phi$.

We show that such an $x$ must always exist. If not, then there exist $f_\xi\in\text{err}(\phi)$ ($\xi<\omega_1$) such that $u(\xi)=e(\phi_{f_\xi},\phi)\backslash\bigcup_{\eta<\xi} e(\phi_{f_\eta},\phi)$ is infinite for every $\xi$. Define $g:\kappa\rightarrow\omega$ by $g(\alpha)=f_\xi(\alpha)$ if $\alpha\in u(\xi)$, $g(\alpha)=0$ otherwise. For some $j<\omega$, $A=\{\xi<\omega_1\,|\,e(\phi_{f_\xi},\phi_g)<\beta_j\}$ is uncountable. For some $k\geq j$, $\{\xi\in A\,|\,u(\xi)\cap L_k\neq \varnothing\}$ is uncountable as well. But this gives uncountably many $\alpha_\xi\in L_k$ such that, for some $i$, $\phi_g(\alpha_\xi,i)=\phi_{f_\xi}(\alpha_\xi,i)\neq\phi(\alpha_\xi,i)=\phi_k(\alpha_\xi,i)$. Hence $\phi_k$ does not trivialize $\Phi\!\!\upharpoonright_{L_k}$, a contradiction.\\

Case 2: $\text{cof}(\kappa)>\omega$. \textit{Stacked functions} are natural attempts to trivialize $\Phi$:
\begin{defin} A collection of functions $f_j\in\omega^\kappa$ such that $\bigcup_{j\in\omega}I_{f_j}=\kappa\times\omega$ is a \textit{stack}. $\phi:\kappa\times\omega\rightarrow\mathbb{Z}$ is \textit{stacked} if $\phi:(\alpha,i)\mapsto\phi_{f_k}(\alpha,i)$ for some stack $\mathcal{F}=\langle f_j\rangle$, where $k=\text{min}\{j:(\alpha,i)\in I_{f_j}\}$.
\end{defin}
If $\mathcal{F}$ so determines $\phi$, write $\phi=\phi^\mathcal{F}$.
\begin{lem} \label{stack} For any stacked functions $\phi$, $\psi$, there exists $\delta<\kappa$ such that $\phi(\alpha,i)=\psi(\alpha,i)$ whenever $\alpha>\delta$.
\end{lem}

\begin{proof} Let $\mathcal{F}=\langle f_j\rangle$, $\mathcal{G}=\langle g_k\rangle$ determine $\phi$ and $\psi$, respectively. Were $e(\phi,\psi)=\{\alpha:\phi(\alpha,i)\neq\psi(\alpha,i)\text{ for some }i\}$ uncountable, so too would be $e(\phi_{f_j},\phi_{g_k})$ for some $j,k\in\omega$. This cannot be; hence $e(\phi,\psi)$ is bounded below $\kappa$.
\end{proof}

Applying the induction hypothesis, for $\beta<\kappa$ fix $\phi_\beta:\beta\times\omega\rightarrow\mathbb{Z}$ trivializing $\Phi\!\!\upharpoonright_{\beta}$. Note that these $\phi_\beta$ ``cohere'': $e(\phi_\beta,\phi_\gamma)$ is finite, for every $\beta<\gamma<\kappa$.
Now fix a stack $\mathcal{F}=\langle f_j\,|\,0<j<\omega\rangle$. Note the index-shift: though $\phi=\phi^\mathcal{F}$ is defined, we've left room at index 0 for one more function $f_0$ (room, in other words, to revise $\phi^\mathcal{F}\!\!\!\upharpoonright_{I_{f_0}}$ to $\phi_{f_0}$). Now assume, towards contradiction, that $\Phi$ is nontrivial.

For all $\alpha<\kappa$ there exists a least $\alpha^+<\kappa$ such that $e(\phi_{\alpha^+},\phi)\cap[\alpha,\alpha^+)$ is infinite; if for some $\beta<\kappa$ this were not so, then
\begin{displaymath}
   \phi'(\alpha,i) = \left\{
     \begin{array}{lr}
       \phi_\beta(\alpha,i) & \alpha<\beta\\
       \phi(\alpha,i) & \text{otherwise}
     \end{array}
   \right.
\end{displaymath}
would trivialize $\Phi$. Let $A=\{\alpha<\kappa\,|\,\alpha\in e(\phi_{\alpha^+},\phi)\}$. If $\alpha\in A$ let $f_0(\alpha)=\text{min}\{i\,|\,\phi_{\alpha^+}(\alpha,i)\neq\phi(\alpha,i)\}$. For $\alpha\in \kappa\backslash A$ let $f_0(\alpha)=0$.

Let $\psi=\psi^{\mathcal{F}\cup\{f_0\}}$; by Lemma~\ref{stack}, take $\delta<\kappa$ such that $\psi(\alpha,i)=\phi(\alpha,i)$ for all $\alpha>\delta$. By the coherence of $\{\phi_\beta\,|\,\beta<\kappa\}$, $\alpha^+=\delta^+$ for $\alpha\in A\cap[\delta,\delta^+)$. So $\phi_{\delta^+}(\alpha,i)\neq\phi(\alpha,i)$ for infinitely many $(\alpha,i)\in I_{f_0}\cap([\delta,\delta^+)\times\omega)$, by the definition of $f_0$. But $\psi(\alpha,i)=\phi_{f_0}(\alpha,i)$ for such $(\alpha,i)$, and $\phi_{f_0}(\alpha,i)=\phi_{\delta^+}(\alpha,i)$ for all but finitely many $(\alpha,i)$, hence $\psi(\alpha,i)\neq\phi(\alpha,i)$ for some $\alpha>\delta$ - a contradiction.
\end{proof}

\section{Open Problems}

The foregoing suggests a number of further questions:

\begin{enumerate}

\item \textit{For $n>1$ does $\text{lim}^n\,\mathbf{A}=0$ imply $\text{lim}^n\,\mathbf{A}_\kappa=0$?}

\item \textit{Does $\text{lim}^n\,\mathbf{A}_\kappa=0$ for all $n>0$, $\kappa\geq\omega$ imply strong homology additive on, e.g., locally compact metric spaces?}

\noindent Here there are two questions, really, in play. Andrei Prasolov has exhibited a paracompact, non-metrizable ZFC counterexample to the additivity of strong homology (see \cite{Pra}). So a first question is \textit{On what class of spaces can strong homology be additive?} Prasolov's example is a kind of upper bound. Secondly: \textit{On that class of spaces, are nonzero $\text{lim}^n\mathbf{A}_\kappa$ the only obstructions to additivity?}

\item \textit{Is it consistent that $\text{lim}^n\,\mathbf{A}_\kappa=0$ for all $n>0$, $\kappa\geq\omega$?}

This extends a question of Moore's (see \cite{Moore}): \textit{Is it consistent that $\text{lim}^1\mathbf{A}=\text{lim}^2\mathbf{A}=0$?}

\item \textit{Is it consistent that $\text{lim}^3\mathbf{A}\neq 0$?}

Arguments like ours for Theorem~\ref{lim2A} would require higher analogues of Theorem~\ref{lim1AF}. An affirmative answer to 4, in other words, would follow from an affirmative answer to 5, in the case $n=2$.

\item \textit{Given an $\omega_n$-chain $\mathcal{F}\subset\mathcal{N}$, does $\text{lim}^n\mathbf{A}^\mathcal{F}\neq 0$?}

\item \textit{Can a witness to $\text{lim}^n\mathbf{A}\neq 0$ be analytic?}

Todorcevic has given a negative answer in the case $n=1$ \cite{T1}.

\end{enumerate}

\subsection*{Acknowledgements}
For suggesting the above problem, and for continual guidance and instruction, the author would like to thank Justin Tatch Moore. The author would like to thank Andrei Prasolov as well, for many helpful comments.

\end{document}